\documentclass[12pt]{amsart}

\usepackage{cite}

\usepackage{amsmath}
\usepackage{amstext}
\usepackage{amssymb}

\usepackage{amsthm}

\theoremstyle{plain}
\newtheorem{theorem}{Theorem}[section]
\newtheorem{lemma}[theorem]{Lemma}
\newtheorem{corollary}[theorem]{Corollary}

\newtheorem{proposition}[theorem]{Proposition}

\theoremstyle{definition}
\newtheorem{definition}{Definition}[section]

\numberwithin{equation}{section}

\DeclareMathOperator{\Rp}{Re}

\newcommand{\ds}{\displaystyle}

\DeclareMathOperator*{\esssup}{ess\,sup}
\DeclareMathOperator*{\essinf}{ess\,inf}

\usepackage{cite}

\begin{document}

\title[Variable Exponent Hardy and Bergman Spaces]%
{Equivalence among various variable exponent Hardy or Bergman spaces}
\author{Timothy Ferguson}
\thanks{Thanks to David Cruz-Uribe for helpful comments}
\address{Department of Mathematics\\University of Alabama\\Tuscaloosa, AL}
\email{tjferguson1@ua.edu}

\date{\today}

\begin{abstract}
  We study the question of when two weighted variable exponent Bergman
  spaces or Hardy spaces are equivalent.
  As an application, we show that variable exponent Hardy spaces have
  a close relation to classical Hardy spaces when
  the exponent is log-H\"{o}lder continuous and has bounded harmonic
  conjugate (when extended from its boundary values to be harmonic in the
  disc). 
  We use this to characterize Carleson measures for these variable exponent
  Hardy spaces.  We also prove under certain conditions
  an analogue of Littlewood subordination and a result on the boundedness
  of composition operators. 
\end{abstract}

%\keywords{}

%\subjclass[2010]{}

%\thanks{}

\maketitle

\section{Introduction}
Let $p(z)$ be a measurable function defined on 
$\mathbb{D}$, the unit disc.
Let $p_- = \essinf p(z)$ and $p_+ = \esssup p(z)$. 
If $0 < p_- \leq p_+ < \infty$
we say that $p(\cdot)$ is a variable exponent on $\mathbb{D}$. 
Let $dA$ denote {\itshape normalized} Lebesgue area measure on the unit disc
and let $dA_\alpha(z) = (\alpha + 1) (1-|z|^2)^\alpha dA(z)$ for
$-1 < \alpha  < \infty$. 
Note that $dA_\alpha(\mathbb{D}) = 1$. 
We say that 
$f \in L^{p(\cdot)}_\alpha$ if 
\[
\rho_{p(\cdot)}(f) := \int_{\mathbb{D}} |f(z)|^{p(z)} \, dA_\alpha(z) < \infty.
\]
We can define the norm of a function $f \in L^{p(\cdot)}_\alpha$ by 
\[
\inf \left\{ \lambda > 0 : 
  \int_{\mathbb{D}} |f(z)/\lambda|^{p(z)} \, dA(z) < 1 \right\}.
\]
This norm is called the Luxemborg-Nakano norm and makes 
$L^{p(\cdot)}_\alpha$ into a Banach space if $p_{-} \geq 1$
(see \cite{DCU_VarLebesgueBook}). 

Recall the inequality $(a+b)^p \leq a^p + b^p$ that holds for 
$a,b \geq 0$ and $0 < p \leq 1$.  
Thus if $0 < p(\cdot) \leq 1$ then as in the constant exponent case the 
distance $d(f,g) = \rho_{p(\cdot)}(f-g)$ defines a translation invariant 
metric (in the sense of addition of vectors, i.e.\ functions) 
under which $L^{p(\cdot)}$ is complete.  Also, since $p_{-} > 0$, the 
Luxemborg-Nakano norm makes $L^{p(\cdot)}$ into a quasi-Banach space 
(which satisfies the axioms for a Banach space except that the triangle 
inequality only needs to hold up to a constant multiple).  
If $0 < p_{-} < 1 < p_{+} < \infty$ then the distance 
$d(f,g) = \rho_{p(\cdot)}[(f-g)\chi_{0 < p(\cdot) < 1}] + 
   \|(f-g)\chi_{p(\cdot) \geq 1}\|_{L^{p(\cdot)}}$ defines a translation 
   invariant metric on $L^{p(\cdot)}$ under which $L^{p(\cdot)}$ is complete.
   With this fact, it is not hard to see that $L^{p(\cdot)}$ is an
   $F$-space, and thus the closed-graph theorem applies to it
   \cite[II.2, Theorem 4]{DunfordSchwartzI}.
Also, the Luxemborg-Nakano norm still makes $L^{p(\cdot)}$ into a 
quasi-Banach space. 

We let $A^{p(\cdot)}_\alpha$ be the subspace of
$L^{p(\cdot)}_\alpha$ consisting of all
analytic functions in $L^{p(\cdot)}_\alpha$.
Since $\|\cdot\|_{L^{p_-}} \leq \|\cdot\|_{L^{p(\cdot)}}$ and
point evaluation is a bounded linear functional on $A^{p_-}$ 
with uniform bound on compact subsets, the
Bergman spaces $A^{p(\cdot)}_\alpha$ are closed.

More generally, we may define the space $A^{p(\cdot)}_{\mu}$ where
$\mu$ is a (not identically zero)
finite measure on the open unit interval by letting
the norm of a function in $L^{p(\cdot)}_\mu$ be
\[
\inf \left\{ \lambda > 0 : 
  \int_{0}^1 \frac{1}{2\pi}
     \int_0^{2\pi} |f(z)/\lambda|^{p(z)} \, d\theta \, d\mu(r) < 1 \right\}.
 \]
 We will not study the $A^{p(\cdot)}_\mu$ spaces in detail,
 so we leave aside the question
 of when $A^{p(\cdot)}_\mu$ is a closed subspace of $L^{p(\cdot)}_\mu$. 
 
We also define the integral mean
$M_{p(\cdot)}(r,f)$ of an function analytic in the unit disc
to be the Luxemborg-Nakano norm of $f$ restricted to the
circle of radius $r$ with normalized Lebesgue measure.  In other words,
\[
M_{p(\cdot)}(r,f) = \inf \left\{ \lambda > 0 : 
  \frac{1}{2\pi} \int_{0}^{2\pi} |f(re^{i\theta})/\lambda|^{p(re^{i\theta})} \,
    d\theta < 1 \right\}.
\]
We also define
\[
  M_{p(\cdot)}^{p(\cdot)}(r,f) =
 \frac{1}{2\pi} \int_{0}^{2\pi} |f(re^{i\theta})|^{p(r e^{i\theta})} \,
 d\theta.
\]
(Note that the superscript $p(\cdot)$ is not really an exponent, but is
suggestive of the notation $M_p^p(r,f)$ used with classical integral
means.)
We let $H^{p(\cdot)}$ be the class of all functions analytic in
$\mathbb{D}$ such that
\[
  \|f\|_{H,p(\cdot),\textrm{sup}} = \sup_{0 \leq r < 1} M_{p(\cdot)}(r,f) < \infty.
\]
For a log-H\"{o}lder continuous exponent $p$, we define 
\[
  \|f\|_{H,p(\cdot),\textrm{lim}} = \lim_{r \rightarrow 1} M_{p(\cdot)}(r,f).
\]
It is a consequence of
\cite[Theorem 5.11]{DCU_VarLebesgueBook}
that the limit in question always exists and
\[
  \|f\|_{H,p(\cdot),\textrm{lim}} \leq
  \|f\|_{H,p(\cdot),\textrm{sup}} \leq
  C  \|f\|_{H,p(\cdot),\textrm{lim}}
\]
for some constant $C$ depending only on $p(\cdot)$ but not on $f$. 
(These two norms are equal in the classical case because the integral
means are increasing). From now on, we will let
$\|f\|_{H,p(\cdot)}$ denote $\|f\|_{H,p(\cdot),\textrm{lim}}$ for
Hardy spaces with log-H\"{o}lder continuous exponent. 
  
We will also let $A^{p(\cdot)}_{-1}$ be an alternate notation for $H^{p(\cdot)}$. 
Note that if $f \in H^{p(\cdot)}$, then $f \in H^{p_-}$, so that
$f$ has nontangential limit almost everywhere.
Furthermore, by Fatou's lemma, if we let
$p(e^{i\theta}) = \liminf_{r \rightarrow 1^-} p(re^{i \theta})$, then
$f(e^{i\theta}) \in L^{p(e^{i\theta})}$. As with Bergman spaces, these
Hardy spaces can be made into $F$-spaces or quasi-Banach spaces, and
they are Banach spaces if $p_{-} \geq 1$.  

We study several questions in this paper.  First, we prove that the
$A^{p(\cdot)}_\alpha$ and $H^{p(\cdot)}$ spaces depend only on the values of
$p$ on the unit circle if $p$ satisfies a condition that is slightly
weaker than log-H\"{o}lder continuity on the closed unit circle.
Log-H\"{o}lder continuity is well known to be a
natural condition when working with variable exponent Lebesgue spaces
(see \cite{DCU_VarLebesgueBook}).

We then discuss the case of variable exponent Hardy spaces in more detail.
In the classical case, it is important that if $f \in H^p$ then
$f = B g$ for the Blaschke product $B$ formed from the zeros of $f$.
Also, $\|f\|_{H,p} = \|g\|_{H,p}$ and  $g^{p/2} \in H^2$ and
$\|g^{p/2}\|_{H,2} = \|g\|_{H,p}^{p/2}$.  We prove an analogous result for
variable exponent Hardy spaces, under the condition that
the exponent is log-H\"{o}lder continuous and that when
restricted to the unit circle and then extended harmonically into the disc,
the exponent has
bounded harmonic conjugate.  This allows us to pass immediately from
certain classical results about Hardy spaces to corresponding results for
these variable exponent Hardy spaces.  Two examples are the
characterization of Carleson measures and an analogue of the Littlewood
subordination theorem.
This allows us to prove that composition operators are
bounded between certain variable exponent Hardy spaces.

In the last section, we discuss a case for the unweighted Bergman space
where the exponent $p$ is assumed to be radial and
$q \leq p(r) \leq \esssup p(r) < \infty$ for some
$q > 0$.  (We do not need to assume any continuity on $p$.) We find a
necessary and sufficient condition for $A^{p(\cdot)}$ to equal $A^q$.
This leads to some interesting and surprising corollaries, such as
Corollary \ref{cor:limsup_p_big}.

\section{Uniformly Radially Log-H\"{o}lder Continuous Exponents}

In this section, we discuss
uniformly radially log-H\"{o}lder continuous exponents
and prove that the Bergman and Hardy spaces they define depend only
on the values of the exponents on the unit circle.  The proof is very
general and uses only the fact that functions in these spaces satisfy
a bound of the form $|f(re^{i\theta})| = O((1-r)^{-a})$ for some
$a > 0$.  We then discuss some applications of the result.  

\begin{definition}
  Let $p$ be an exponent on 
  $\overline{\mathbb{D}}$ with $0 < p_- \leq p_+ < \infty$ and 
  let $0 \leq R < 1$ be fixed.  
  Suppose that there is an absolute constant $C$ such that for each fixed 
  $\theta$, one has
  \[
    |p(re^{i\theta}) -  p(e^{i\theta})| \leq \frac{C}{-\log(1-r)}
  \]
  for $R \leq r < 1$. 
We then say that $p$ is uniformly radially log-H\"{o}lder continuous. 
\end{definition}

\begin{definition}
Let $p$ be an exponent on 
$\overline{\mathbb{D}}$ with $0 < p_- \leq p_+ < \infty$.
We say that $p$ is log-H\"{o}lder continuous if there is a constant
$C$ such that for all $z,w \in \overline{\mathbb{D}}$ with $|z-w| < 1$
we have
\[
  |p(z)-p(w)| \leq \frac{C}{\log \frac{1}{|z-w|}}.
  \]
\end{definition}
A slight problem with the above definition is that the right hand side
of the inequality becomes infinite as $|z-w| \rightarrow 1$.  However, we
still get a uniform bound on $|z-w|$ by applying the inequality first to
$z$ and $(2z+w)/3$ and then to $(2z+w)/3$ and $(z+2w)/3$, and
finally to $(z+2w)/3$ and $w$. An alternate route
is to note that we can change the inequality in the definition to
\[
  |p(z)-p(w)| \leq \frac{C'}{\log \frac{a}{|z-w|}}.
  \]
  for any $a > 2$ and any $z, w \in \overline{\mathbb{D}}$ without
  affecting the class of log-H\"{o}lder continuous functions.
  Furthermore, $C'$ can be bounded by a number depending on
  $C$ and $a$, and $C$ can be bounded by a number depending on
  $C'$ and $a$. 
  
\begin{theorem}\label{thm:logholder_equiv}
  Let $-1 \leq \alpha < \infty$. 
  Suppose that $p(z)$ is a uniformly radially log-H\"{o}lder
  continuous exponent and  
  let $\widehat{p}(re^{i\theta}) = p(e^{i\theta})$.
  Then  
$A_\alpha^{p(\cdot)} = A_\alpha^{\widehat{p}(\cdot)}$. 
\end{theorem}
\begin{proof}
First let $-1 < \alpha < \infty$. 
  Suppose that $f \in A^{p(\cdot)}_\alpha$.  
We may assume without loss of generality that it has unit norm.
Then $f \in A^{p_-}_\alpha$ so from 
\cite{Dragan_point_eval},
\[
|f(re^{i\theta})| \leq \frac{1}{(1-r)^{(2+\alpha)/p_-}}.
\]
Note that
\[
(1-r)^{-|p(re^{i\theta}) - p(e^{i\theta})|} \leq 
(1-r)^{C/\log(1-r)} = e^C
\]
Thus if $|f(z)| \geq 1$ we have that 
\[
  |f(re^{i\theta})|^{p(re^{i\theta}) - p(e^{i\theta})} \leq e^{(2+\alpha)C/p_-}.
\]
Thus 
\[
\int_{\{|f(z)| \geq 1\}} |f(z)|^{\widehat{p}(z)} dA_\alpha(z) \leq 
e^{(2+\alpha)C/p_-}
\int_{\{|f(z)| \geq 1\}} |f(z)|^{p(z)} dA_\alpha(z).
\]
Since %
$\int_{\{|f(z)| < 1\}} |f(z)|^{\widehat{p}(z)} dA_\alpha(z)$
is bounded by $1$, we have that 
\[
  \rho_{\widehat{p}}(f) \leq 
  e^{2C/p_-} \rho_{p}(f) + 1 = e^{2C/p_-} + 1.
\]
Thus for any $f \in A^{p(\cdot)}$, one has that 
$\|f\|_{A,{\widehat{p}(\cdot)}} \leq (e^{(2+\alpha)C/p_-} + 1) \|f\|_{A,{p(\cdot)}}$. 
A similar argument shows that 
$\|f\|_{A,{p(\cdot)}} \leq (e^{(2+\alpha)C/\widehat{p}_-} + 1) \|f\|_{A,{\widehat{p}(\cdot)}}$. 

For the case $\alpha = -1$, performing similar calculations shows that
$M_{p(\cdot)}(r,f) \leq (e^{C/p_-}+1) M_{\widehat{p}(\cdot)}(r,f)$ and that
$M_{p(\cdot)}(r,f) \leq (e^{C/\widehat{p}_-}+1) M_{\widehat{p}(\cdot)}(r,f)$,
which
gives the result.%
\end{proof}

We need the following result found in \cite[Lemma 1]{tjf:bergprojbounds}.
\begin{lemma}\label{lemma:poissonbound}
  Let $p>1$ and $0<r<1$. Then 
\[
\frac{1}{2\pi} \int_0^{2\pi} \frac{1}{|1-re^{i\theta}|^p} d\theta %
\le 
\frac{\Gamma(p-1)}{\Gamma(p/2)^2} (1-r^2)^{1-p}.
\]
Furthermore, the bound is sharp, in the sense that the integral in question, 
divided by $(1-r^2)^{1-p}$, is always less than or equal to
$\tfrac{\Gamma(p-1)}{\Gamma(p/2)^2}$, but the quotient approaches 
$\tfrac{\Gamma(p-1)}{\Gamma(p/2)^2}$ as $r \rightarrow 1$. 
\end{lemma}

\begin{theorem}
Suppose that $p$ and $q$ are two exponents in $\mathbb{D}$ 
 with 
$0 < p_{-}, q_{-}$ and $p_{+}, q_{+} < \infty$ and that 
$\essinf_{z \in D \cap \mathbb{D}} p(z) >
  \esssup_{z \in D \cap \mathbb{D}} q(z)$ where $D$ is a disc
centered on the boundary 
of the unit disc.  
Let $-1 \leq \alpha < \infty$. 
Then $A^p_\alpha \neq A^q_\alpha.$
\end{theorem}

\begin{proof}
  Assume without loss of generality that $D$ is centered at $1$. 
  Let $s$ be such that
  $p_{-,D} = \essinf_{z \in D} p(z) > s > \esssup_{z \in D} q(z) = q_{+,D}.$
  It follows from the lemma that 
  the function $f(z)= (1-z)^{-(2+\alpha)/s}$ is in $A^{q_{+,D}}$ but not in
  $A^{p_{-,D}}$. 
  Thus $f(z) \in A^{q(\cdot)}$ but
  $f \not \in A^{p(\cdot)}$. 
\end{proof}

\begin{corollary}
If $p$ and $q$ are two exponents continuous in $\overline{\mathbb{D}}$ with 
$0 < p_{-}, q_{-}$ and $p_{+}, q_{+} < \infty$ then if 
$p(e^{i\theta}) \neq q(e^{i\theta})$ for some $\theta$, then 
$A^p_\alpha \neq A^q_\alpha$.
\end{corollary}

\begin{theorem}
Suppose that $p$ and $q$ are two exponents in 
$\overline{\mathbb{D}}$ with $0 < p_- \leq p_+ < \infty$ and 
$0 < q_- \leq q_+ < \infty$ 
and assume that $p$ and $q$ are log-H\"{o}lder
continuous. 
Then $A^p = A^q$ if and only if 
$p(e^{i\theta}) = q(e^{i\theta})$ for all $\theta$. 
\end{theorem}

\section{Complex Exponents and Hardy Spaces}
In this section we find a condition under which we can use
complex variable exponentials to easily derive many results for
certain variable exponent Hardy spaces.  The following
proposition lets us simply our definitions somewhat.

\begin{proposition}
If a continuous harmonic function has log-H\"{o}lder continuous boundary
values on the unit circle, then it is log-H\"{o}lder continuous in the
closed unit disc.  
\end{proposition}
\begin{proof}
  Let $f$ be the harmonic function.  It is enough to show that
  $|f(re^{i\theta}) - f(r)| \leq C/\log|2\pi/\theta|$ for any $r$ and
  that 
  $|f(1)-f(r)| \leq -C/\log|1-r|$.  The first inequality
  follows easily from the
  fact that
  \[
    f(re^{i\theta})-f(r) =
    \frac{1}{2\pi} \int_{0}^{2\pi} P_r(e^{i\psi})
      [f(e^{i(\theta-\psi)})-f(e^{-i\psi})]
    \, d \psi
  \]
  so that
  \[
    |f(re^{i\theta})-f(r)|
    \leq
    \frac{1}{2\pi} \int_{-\pi}^{\pi} P_r(e^{i\psi}) \frac{C}{\log(2\pi/|\theta|)} 
    \, d \psi \leq
     \frac{C}{\log(2\pi/|\theta|)}.
  \]
  for $|\theta| < 1$.

  To prove the second inequality, first note that it is enough to show that
\[
  |f(1) - f(r)| \leq \frac{C}{\log \frac{1}{1-r}}
\]
for $r$ close enough to $1$.  But
\[
  f(1) - f(r) =
   \frac{1}{2\pi}
  \int_{-\pi}^{\pi} \frac{1}{2\pi} P_r(e^{it})(f(1) - f(e^{it}))
     \, dt
\]
Note that 
\[
  P_r(t) \leq \frac{C(1-r)}{(1-r)^2 + t^2}
\]
for some constant $C$ (see \cite[p.\ 74]{D_Hp}).  
Thus we want to show that
\[
  \int_{-\pi}^{\pi} \frac{1-r}{(1-r)^2 + t^2} \frac{|\log(1-r)|}{\log|\pi e^2/t|}
  \, dt
  \leq C
\]
Now let $u = t/(1-r)$ to rewrite this integral as
\[
  \int_{-\pi/(1-r)}^{\pi/(1-r)} \frac{du}{1+u^2}
  \frac{|\log(1-r)|}{|\log((1-r)|u|/(\pi e^2))|} \, du
\]
Note that the integrand is even and that
\[
   \int_{0}^{1} \frac{du}{1+u^2}
  \frac{|\log(1-r)|}{|\log((1-r)|u|/(\pi e^2)|} \, du
\]
is bounded independently of $r$ for $1-r$ small.
Let $a=1-r$ and consider the integral
\begin{equation}\label{eq:logestint}
  \int_1^{\pi/a} \frac{1}{u^2} \frac{|\log a|}{|\log(au/(\pi e^2))|} \, du.
\end{equation}
We may assume without loss of generality that $a < \pi/e$.  
Let $n$ be the first integer such that $\pi/a \leq e^n$,
so that in particular $n \geq 1$. Let $j+1 \leq n$. 
Notice that 
\[
  \int_{e^j}^{e^{j+1}} \frac{1}{u^2} \frac{|\log a|}{|\log(au/(\pi e^2))|} \, du
  \leq e^{-j} \frac{|\log a|}{-\log a + 1 + \log \pi - j}.
\]
Then we have that
the above expression is at most
\[
  e^{-j} \frac{n}{n  - j} %
  \leq (j+1)e^{-j}.
\]
But since $\displaystyle \sum_{j=0}^\infty (j+1)e^{-j}$ converges, the
integral in \eqref{eq:logestint} is bounded by a number independent of $a$,
which finishes the proof. 
\end{proof}

\begin{definition}\label{def:logholderbc}
Suppose that $p(\cdot)$ is log-H\"{o}lder continuous.
By its harmonic conjugate we mean the harmonic conjugate of the 
harmonic function with the same boundary values as $p(\cdot)$.
We say that an exponent $p$ is log-H\"{o}lder continuous with bounded 
conjugate if $p$ is log-H\"{o}lder continuous
in the closed unit disc and its harmonic conjugate is bounded.
\end{definition}
Note that if $p$ is log-H\"{o}lder continuous in the closed unit
disc then so is the harmonic function with the same boundary values
as $p$, by the above Proposition. 
The main advantage of working with exponents satisfying this definition
is that for such exponents, many results
about classical Hardy spaces immediately imply corresponding results for
variable Hardy spaces.

Definition \ref{def:logholderbc} is only slightly more restrictive than
log-H\"{o}lder continuity.  For example, if $p$ is H\"{o}lder
continuous then $p$ is log-H\"{o}lder continuous with bounded conjugate.
This also happens if $p$ satisfies a condition like log-H\"{o}lder
continuity but with an exponent greater than $1$ on the logarithm. 

We now discuss some basic facts about variable exponent Hardy spaces.
For more information see
\cite{Kokilashvili-Paatashvili2006, Kokilashvili-Paatashvili2008, Kokilashvili-Paatashvili2015, DCU-VariableHardy}. 
If $p(\cdot)$ is log-H\"{o}lder continuous and $p_- \geq 1$, then
$\|f(e^{i\theta})\|_{L^p(e^{i\theta})} = \|f\|_{H,p(\cdot)}$.
Furthermore, we have mean convergence to boundary values. 
Both of these facts follow from \cite[Theorem 5.11]{DCU_VarLebesgueBook}. 

If $0 < p_- < 1$, we may write $f = B g$ where $g$ is the
Blaschke product of the zeros of $f$.  Then $g^{p_{-}} \in H^{p/p_{-}}$,
so
\[
  M_{p(\cdot)}(r,g)^{p_{-}} = M_{p(\cdot)/p_{-}}(r,g^{p_{-}})
  \rightarrow
  \|g(e^{i \cdot})^{p_{-}}\|_{L^{p(\cdot)/p_{-}}} = 
  \|g(e^{i\cdot})\|^{p_{-}}_{L^{p(e^{i\cdot})}}
\]
as $r \rightarrow 1$.
  But also $g(re^{i\theta}) \rightarrow g(e^{i\theta})$ pointwise almost
  everywhere, so by an application of Egorov's theorem, we have that
  $g$ converges to its boundary values in the mean
  (see \cite[Lemma 1 after Theorem 2.6]{D_Hp}).  The fact that
  $f$ converges to its boundary values in the mean follows from
  Fatou's lemma and 
  a similar
  application of Egorov's theorem. 
   (see \cite[Proof of Theorem 2.6]{D_Hp}). 
  Furthermore,
  $\|f\|_{H, p(\cdot)} = \|g\|_{H,p(\cdot)}$.

  We will not need the following Theorem, but we include it because it
  is simpler to prove than Theorem \ref{thm:boundedargdecomp_psmall}.
  It may generally be used in place of that theorem when $p_{-} > 1$.
  For the remainder of this section,
  we will let $\widetilde{p}$ be the harmonic conjugate of the
  extension of
  $p$ as a harmonic function in the unit disc, and let
  $\widehat{p} = p + i \widetilde{p}$.

\begin{theorem}\label{thm:boundedargdecomp_pbig}
  Suppose that $f \in H^p$ and that $p$ is log-H\"{o}lder continuous
  with bounded conjugate.
  Assume that $p$ is harmonic in $\mathbb{D}$. 
  Also suppose that
  $1 < p_{-} \leq p_{+} < \infty$. Then we may write
  \[
    f = f_1 - f_2
  \]
  where each $f_j$ satisfies either the inequality
  $-\tfrac{\pi}{2} < \arg f_j(z) < \tfrac{\pi}{2}$ for all
  $z \in \mathbb{D}$, or satisfies
  $f_j(z) = 0$ for all $z \in \mathbb{D}$.
  Furthermore we have
  $\|f_j\|_{H,p(\cdot)} \leq C \|f\|_{H,p(\cdot)}$ for some constant
  $C$ depending only on $p_{-}$, $p_{+}$,
  and the log-H\"{o}lder constant of
  $p$.
  Also, 
  \begin{equation}\label{eq:complexpbd_pbig}
    \exp(-\pi  \|\widetilde{p}\|_\infty/2) |f_j|^{p(z)}
    \leq |f_j(z)^{\widehat{p(z})/2}|^2
    \leq \exp(\pi  \|\widetilde{p}\|_\infty/2) |f_j|^{p(z)}
  \end{equation}
If $\|f\|_{H,p(\cdot)} \leq 1/C$ then 
\[
  \|f_j^{\widehat{p}(\cdot)/2}\|_{H,2} \leq 
        \exp(\pi \|\widetilde{p}\|_\infty/4)
\]
where $C$ is the same constant as above.
If
$\|f_j\|_{H,p(\cdot)} \leq C'$
for some constant $C'$ 
  then
  \[
    \|f_j^{\widehat{p}(\cdot)/2}\|_{H,2} \leq
     \max(C'^{p_-/2},C'^{p_+/2}) \exp(\pi \|\widetilde{p}\|_\infty/2).
  \]
  A converse also holds:
  if $\|f_j^{\widehat{p}(\cdot)/2}\|_{H,2} \leq C'$ then
  \[
    \|f_j\|_{H,p(\cdot)} \leq
    \max(C'^{2/p_-},C'^{2/p_+}) \exp(\pi \|\widetilde{p}\|_\infty/2).
  \]
\end{theorem}
\begin{proof}
  Let $f_1$ be the analytic function whose real part equals
  $\max(f,0)$ on the unit circle, and let
  $f_2$ be the analytic function whose real part equals
  $-\min(f,0)$ on the unit circle. 
  Then $f = f_1 -  f_2$.  Also
  $\Rp f_j \geq 0$ and 
  \[
    \|f_j\|_{H^{p(\cdot)}} \leq C \|f\|_{H^{p(\cdot)}} 
  \]
  since $p$ is log-H\"{o}lder continuous and thus the harmonic conjugation
  operator is bounded on $H^{p(\cdot)}$
  \cite[Theorem 5.39]{DCU_VarLebesgueBook}.  

  Note that 
  \[
    |f_j(z)^{\widehat{p}(z)/2}| =
    |f_j(z)|^{p(z)/2} e^{- \widetilde{p}(z) \arg f_j(z)/2}
  \]
  Thus equation \eqref{eq:complexpbd_pbig} holds. 
  If $\|f\|_{H,p(\cdot)} \leq 1/C$ then $\|f_j\|_{H,p(\cdot)} \leq 1$, so
  $\limsup_{r \rightarrow 1} M_{p(\cdot)}(r,f_j) \leq 1$ which implies that
  \[
    \limsup_{r \rightarrow 1}
        \frac{1}{2\pi} \int_0^{2\pi} |f_j(re^{i\theta})|^{p(re^{i\theta})} \,
    d\theta \leq 1.
  \]
  But this implies that
  \[
    \lim_{r \rightarrow 1} \frac{1}{2\pi} \int_0^{2\pi}
    |f_j(re^{i\theta})^{\widehat{p}(re^{i\theta})}|^2 d\theta \leq
      \exp(\pi \|\widetilde{p}\|_\infty/2).
    \]
    Note that the limit above exists because the $M_2$ integral means are
    increasing.  The rest of the proof follows from the
    definition of the Luxemborg-Nakano norm and
    equation \eqref{eq:complexpbd_pbig}. 
\end{proof}

The following theorem is similar to the previous one, but also applies in
the case that $p_{-} < 1$. 
\begin{theorem}\label{thm:boundedargdecomp_psmall}
  Suppose that $f \in H^p$ and that $p$ is log-H\"{o}lder continuous
  with bounded conjugate.  Also assume that $p$ is harmonic.
  Also suppose that
  $0 < p_{-} \leq p_{+} < \infty$. Let $n$ be the least integer
  such that $np_{-} \geq 2$. Then we may write
  \[
    f = B \sum_{j=1}^{n+1} f_j
  \]
  where $B$ is the Blaschke product of the zeros of $f$ and
  each $f_j$ satisfies either the inequality
  $-\tfrac{\pi n}{2} < \arg f_j(z) < \tfrac{\pi n}{2}$ for all
  $z \in \mathbb{D}$, or satisfies
  $f_j(z) = 0$ for all $z \in \mathbb{D}$.
  Furthermore we have
  $\|f_j\|_{H,p(\cdot)} \leq C \|f\|_{H,p(\cdot)}$ for some constant
  $C$ depending only on $n$, $p_{+}$, and on the log-H\"{o}lder constant of
  $p$.
  Also,
  \begin{equation}\label{eq:complexpbd}
    \exp(-\pi n \|\widetilde{p}\|_\infty/2) |f_j|^{p(z)}
    \leq |f_j(z)^{\widehat{p(z})/2}|^2
    \leq \exp(\pi n \|\widetilde{p}\|_\infty/2) |f_j|^{p(z)}. 
  \end{equation}
If $\|f\|_{H,p(\cdot)} \leq 1/C$ then
\[
  \|f^{\widehat{p}(\cdot)/2}\|_{H,2} \leq 
        \exp(\pi n \|\widetilde{p}\|_\infty/4)
\]
where $C$ is the same constant as above.
If
$\|f_j\|_{H,p(\cdot)} \leq C'$
for some constant $C'$ 
  then
  \[
    \|f_j^{\widehat{p}(\cdot)/2}\|_{H,2} \leq
     \max(C'^{p_-/2},C'^{p_+/2}) \exp(\pi n \|\widetilde{p}\|_\infty/2).
  \]
  A converse also holds:
  if $\|f_j^{\widehat{p}(\cdot)/2}\|_{H,2} \leq C'$ then
  \[
    \|f_j\|_{H,p(\cdot)} \leq
    \max(C'^{2/p_-},C'^{2/p_+}) \exp(\pi n \|\widetilde{p}\|_\infty/2).
  \]
\end{theorem}
\begin{proof}
  Let $f = Bg$ where $B$ is the Blaschke product formed with the
  zeros of $f$.  Then
  $g^{1/n} \in H^{np(\cdot)}$.  Thus we may write $g^{1/n} = g_1 -  g_2$, where
  $\Rp g_j \geq 0$ and 
  \[
    \|g_j\|_{H^{np(\cdot)}} \leq C \|g^{1/n}\|_{H^{np(\cdot)}} = 
    C \|f\|_{H^{p(\cdot)}}^n
  \]
  since $p$ is log-H\"{o}lder continuous
  and thus the harmonic conjugation
  operator is bounded on $H^{p(\cdot)}$.
  We will assume that no $g_j$ is identically zero (the proof where one is 
  identically zero is nearly identical but easier). 
  Since $g_j(z) > 0$ for all $z \in \mathbb{D}$, 
  we may define $\arg g_j(z)$ so that
  $-\tfrac{\pi}{2} < \arg g_j(z) < \tfrac{\pi}{2}$.
  Thus $g(z) = (g_1 - g_2)^n$, and so $g(z)$ can be written as a
  sum of $n+1$ functions, say $f_1, \ldots, f_{n+1}$, such that
  each one satisfies $-\tfrac{-\pi n}{2} < \arg f_j < \tfrac{\pi n}{2}$.
  Also, each $f_j$ satisfies $\|f_j\|_{H^{p(\cdot)}} \leq C \|f\|_{H^{p(\cdot)}}$ 
  by H\"{o}lder's inequality for variable Lebesgue spaces
  \cite[Corollary 2.28]{DCU_VarLebesgueBook}. 
  Then
  \[
    |f_j(z)^{\widehat{p}(z)/2}| =
    |f_j(z)|^{p(z)/2} e^{- \widetilde{p}(z) \arg f_j(z)/2}
  \]
    Thus equation \eqref{eq:complexpbd} holds. 
    The rest of the proof follows as in
    Theorem \ref{thm:boundedargdecomp_pbig}. 
  \end{proof}

We now discuss several results that follow quickly from the
corresponding results for classical Hardy spaces.
The first deals with Carleson measures.  We say that a measure $\mu$
on the unit disc is a Carleson measure for $H^p$, where $0 < p < \infty$, if
$\|f\|_{L^p(\mu)} \leq C \|f\|_{H^p}$ for all $f \in H^p$.  It is well known that
Carleson showed that a necessary and sufficient condition for a measure to be a
Carleson measure for $H^p$ is that $\mu(S) \leq C' h$ for every Carleson square
\[S_{h,\theta_0} = \{z: 1-h \leq |z| < 1 \text{ and }
  \theta_0 - h/2 < \arg z < \theta_0 + h/2\} \]
(see \cite{MR0141789,MR0117349}.)

Duren in \cite{D_Carlesonmeasure} generalized this.  His method is based
on H\"{o}rmander's proof in \cite{MR0234002}
of the Carleson measure theorem.  Let $0 < p \leq q < \infty$ and let
$\mu$ and $S$ be as before.  Then a necessary and sufficient condition for
$\|f\|_{L^q(\mu)} \leq C \|f\|_{H^p}$ to hold
for every $H^p$ function $f$ is that
$\mu(S) \leq C' h^{q/p}$ for every Carleson square $S$. 
We generalize this result to variable exponent Hardy spaces.
\begin{theorem}
Suppose that
  that $0 < p_{-} \leq p(\cdot) \leq p_{+} < \infty$ and let $1 \leq a < \infty$ be a
  constant.
Let $p$ be log-H\"{o}lder continuous with bounded conjugate. 
Then a necessary and sufficient condition that there exist a constant
$C$ such that 
$\|f\|_{L^{ap(\cdot)}(\mu)} \leq C \|f\|_{H,p(\cdot)}$ for every
function $f$ in $H^{p(\cdot)}$ is that
  $\mu(S) < C' h^{a}$ for every Carleson square $S$.
  \end{theorem}
\begin{proof}
  The proof of the necessity uses a testing function similar to the one used in
  \cite{MR0141789}
  but some difficulties
  arise because of the variable exponents.
  The calculation involved is similar to one in \cite{XuanThesis}. 
  Let $|z_0| = \rho$ and
  let $h = 1 - \rho$.
  It is enough to show the result holds when $h$ is near $0$. 
  Let $q(z) = a p(z)$.  We may assume that the support of $\mu$ is in $S$. 
  Consider $f(z) = (1-\overline{z_0} z)^{-2/\widehat{p}(z)}$.
  Note that $1-z_0 z$ has bounded argument and so $f(z)$ is well defined and
  \[
    |f(z)| \asymp \frac{1}{|1-\overline{z_0}z|^{2/p(z)}}.
  \]
  Thus
  \[
    \frac{1}{2\pi}\int_{0}^{2\pi} |f(z)|^{p(z)} d\theta \asymp
     \frac{1}{2\pi}\int_{0}^{2\pi} |1-\overline{z_0}z|^2 d\theta
     = (1-\rho^2)^{-1},  \]
   where $A \asymp B$ means the quantities $A$ and $B$ are equivalent
   up to a multiplicative constant. 

  Now let \[g(z) = (1-\rho^2)^{1/p(\theta_0)} f(z).\]
  For $\rho$ near to $1$, we have that
  $|1-\overline{z_0}z| \asymp d$, where $d$ is the distance from $z$  to
  $1/\overline{z_0}$.  It is also clear that
  $d \asymp |\theta-\theta_0| + h$, where $\theta = \arg(z)$ and
  $\theta_0 = \arg(z_0)$.  Thus
  $g(z) \asymp h^{1/p_0} d^{-2/p(z)}$, where
  $p_0 = p(e^{i\theta_0})$. 

  It follows that $|g(z)|^{p(z)} \leq C$ for
  $|\theta-\theta_0|  \geq C h^{p_{-}/2p_{0}}$.
  Recall that in a disc $D$, if $p(\cdot)$ is log-H\"{o}lder
continuous then $|D|^{(p_{+,D)}-(p_{-,D})} \leq C$ for some
constant $C$, where $|D|$ is the area of $D$, and
$p_{+,D}$ and $p_{-,D}$ are the maximum and
minimum, respectively, of $p(\cdot)$ restricted to $D$. 
  For $|\theta-\theta_0| < C h^{p_{-}/2p_{0}}$ we thus have that
    \[
    h^{p(z)/p_{0}} \leq C h^a
  \]
  by the log-H\"{o}lder continuity of $p$. 
  Thus $\|g\|_{H,p(\cdot)} \leq C$. 
  The Carleson measure condition applied to 
  $g$ implies that
  \[
    \int_{S} (1-\rho^2)^{ap(z)/p_{0}}
      |(1-\overline{z_0} z)|^{-2a} \, d\mu(z) \leq C
  \]
  But this implies that
  \[
    \mu(S) h^{-2a} \leq 
    \int_{S} 
    |(1-\overline{z_0} z)|^{-2a} \, d\mu(z) \leq
    C(1-\rho^2)^{-ap_{+,S}/p_{-,S}}.
  \]
  But by log-H\"{o}lder continuity, $(1-\rho^2)^{-ap_{+,S}/p_{-,S}}$ is bounded by
$Ch^{-a}$. 
Now use the fact that in $S$ we have that
$(1-\overline{z_0}z)^{-1} \geq C h^{-1}$ (see \cite[Proof of Theorem 9.3]{D_Hp})
to see that $\mu(S) \leq C h^a$. 

For the sufficiency
let the $f_j$ be defined as in Theorem \ref{thm:boundedargdecomp_psmall}
and assume without loss of generality that $\|f\|_{H,p(\cdot)} \leq 1$. 
  Then
  $\|f_j^{\widehat{p(z)}/2}\|_{H^2} \leq C$,  
  which implies that 
  $f_j^{\widehat{p(z)}/2} \in L^{2a}(\mu)$ and
  $\|f_j^{\widehat{p(z)}/2}\|_{L^{2a}(\mu)} \leq C$.
    And since 
  $|f_j(z)|^{p(z)/2} \leq C |f_j(z)^{\widehat{p}(z)/2}|$, this shows that 
  $\|f_j(z)\|_{L^{ap(\cdot)}} \leq C$, so that 
  $\|f(z)\|_{L^{ap(\cdot)}} \leq C$. 
\end{proof}

The following two corollaries follow immediately.
The sharp form of the classical case of the first can be found in
\cite{Dragan_isoperimetric}. 
\begin{corollary}
  Suppose that $p$ is an exponent with $0 < p_- \leq p_+ < \infty$.
  Also suppose that $p$ is log-H\"{o}lder continuous with bounded 
  conjugate.  
  If $f \in H^{p(\cdot)}$ then $f \in A^{2p(\cdot)}$ and 
  $\|f(z)\|_{A^{2p(\cdot)}} \leq C \|f(z)\|_{H^{p(\cdot)}}$.
\end{corollary}

The classical Fejer-Riesz inequality (see \cite[Theorem 3.13]{D_Hp})
says that if $f \in H^p$ then
\[
  \int_{-1}^1 |f(x)|^p \, dx \leq
  \frac{1}{2} \int_0^{2\pi} |f(e^{i\theta})|^p \, d\theta.
\]
We have the following analogy. 
\begin{corollary}%
  Let $f \in H^{p(\cdot)}$ where $p$ is log-H\"{o}lder continuous with bounded
  conjugate.  Then
  \[
    \|f\|_{L^{p(\cdot)}((-1,1))} \leq C \|f\|_{H,{p(\cdot)}}
  \]
  for some constant $C$.  Here we may define $p(-x) = p(-1)$ and
  $p(x) = p(1)$ for $x > 0$. 
\end{corollary}

We now give a variable exponent form of another classical result, which has
applications to composition operators. 
Let $f$ and $F$ be two analytic functions in the unit disc.  Recall that
$f$ is subordinate to $F$, written $f \prec F$, if $f(z) = F(\omega(z))$
where $\omega: \mathbb{D} \rightarrow \mathbb{D}$ and
$|\omega(z)| \leq |z|$. If we wish to specify $\omega$, we will write
$f \prec_\omega F$.

In this theorem, we take $p(\cdot)$ to be harmonic.  The reason is that
$p(\omega(\cdot))$ may not be log-H\"{o}lder continuous, so we need to
specify its values inside the unit disc. 
\begin{theorem}[Littlewood Subordination] Suppose that
  $f \prec_\omega F$ and that $F \in H^{p(\cdot)}$, where
  $p$ is harmonic and log-H\"{o}lder continuous with bounded conjugate.
  Then $f \in H^{p(\omega(\cdot))}$ and
  $\|f\|_{H,p(\omega(\cdot))} \leq C \|F\|_{H,p(\cdot)}$.
  Furthermore,
  \[
    M_{p(\omega(\cdot))}(r,f) \leq C M_{p(\cdot)}(r,F).
  \]
\end{theorem}
\begin{proof}
  Assume without loss of generality that $\|F\|_{H,p(\cdot)} = 1$. 
  Decompose $F$ as in Theorem \ref{thm:boundedargdecomp_psmall} so that
  \[
    F = B \sum_{j=1}^{n+1} F_j.
  \]
  Then
  \[
    f(z) = B(\omega(z)) \sum_{j=1}^{n+1} F_j(\omega(z)).
  \]
  Recall that $|B(\omega(z))| < 1$. 
  Also 
  $\|F_j(\omega(\cdot))^{\widehat{p}(\omega(\cdot))/2}\|_{H,2} \leq
  \|F_j(\cdot)^{\widehat{p}(\cdot)/2}\|_{H,2}$ by the Littlewood subordination
  theorem for $H^2$.  Since $F_j(\omega(z))$ has bounded argument,
  for $C' = \exp(\|\widetilde{p}(\cdot)\|_\infty \|\arg F_j\|_\infty /2)$ we have
  that
\[
  (1/C') |F_j(\omega(z))|^{p(\omega(z))} \leq
    |F_j^{\widehat{p}(\omega(z))/2}|^2 \leq
    C' |F_j(\omega(z))|^{p(\omega(z))}.
  \]
  This implies that
  $\|F_j(\omega(\cdot))\|_{H,p(\omega(\cdot))}
  \leq C$ %
  for some constant $C$. 
  And thus
  \[
    \|f\|_{H,p(\omega(\cdot))} \leq
    \sum_{j=1}^{n+1} \|F_j(\omega(\cdot))\|_{H,p(\omega(\cdot))}
    \leq C
    \]
    for a different constant $C$.
    This implies in general that
    $\|f\|_{H,p(\omega(\cdot))} \leq C \|F\|_{H,p(\cdot)}$, even if we
    drop the assumption that $\|F\|_{H,p(\cdot)} = 1$. 
    This proves the first part of the theorem.
    The second part follows from applying the first part to the functions
    $f(rz)$ and $F(rz)$. 
  \end{proof}

  \begin{corollary}[{see \cite[Theorem 3.1]{Cowen-Maccluer_Comp}}]
    Let $\mu$ be a finite positive measure on the closed unit interval and
    let $p(\cdot)$ and $\omega$ be as above.
    Then the composition operator $C_\omega$ is bounded
    from $A^{p(\cdot)}_\mu$ to
    $A^{p(\omega(\cdot))}_{\mu}$, where $C_\omega(f)(z) = f(\omega(z))$. 
  \end{corollary}
  \begin{proof}
    Suppose that $f$ has unit $A^{p(\cdot)}_{\mu}$ norm. 
    For each $0 \leq r < 1$ we have by the above theorem that
    \[
      \int_0^{2\pi}
      |f(\omega(re^{i\theta}))|^{p(\omega(r e^{i\theta}))} d\theta
      \leq
      C.
    \]
    Integration with respect to $d\mu(r)$ now gives the result. 
\end{proof}

    \begin{corollary}
    Let $\phi: \mathbb{D} \rightarrow \mathbb{D}$, and let
    $C_\phi$  be the composition operator defined by
    $(C_\phi f)(z) = f(\phi(z))$.
    Let $p$ be a harmonic log-H\"{o}lder continuous exponent with bounded
    conjugate such that $0 < p_{-} \leq p_{+} < \infty$. 
    Then $C_\phi$ is bounded from
    $A^{p(\cdot)}_\alpha$ to $A^{p(\phi(\cdot))}_\alpha$ 
    for $-1 \leq \alpha < \infty$.
  \end{corollary}
  \begin{proof}
    Suppose that $\phi(0) = \lambda$. 
    We may write $\phi = \phi_\lambda \circ \omega$, where
    $\omega: \mathbb{D} \rightarrow \mathbb{D}$ and
    $\omega(0) = 0$, and
    \[
      \phi_\lambda(z) = \frac{\lambda-z}{1 - \overline{\lambda}z}.
    \]
    Note that $\phi_\lambda$ is its own inverse.
    The previous corollary implies that $C_\omega$ is bounded from
    $H^{p(\phi_\lambda(\cdot))}$ to
    $H^{p(\omega(\phi_\lambda(\cdot)))} = H^{p(\phi(\cdot))}$.
    We will show that $C_{\phi_\lambda}$ is bounded from
    $H^{p(\cdot)}$ to $H^{p(\phi_\lambda(\cdot))}$. 
    Note that (see \cite[p.\ 38]{D_Ap})
    \[
      1 - |\phi_\lambda(z)|^2 =
      \frac{(1-|\lambda|^2)(1-|z|^2)}{(1-\overline{\lambda}z)^2}
    \]
    and $|\phi_\lambda'(z)|$ is bounded above and below in the
    unit disc. 
    Thus 
    \[
      \begin{split}
        &\phantom{{}={}}\int_{\mathbb{D}} |f(\phi_\lambda(z))|^{p(\phi_\lambda(z))}
        (1-|z|^2)^{\alpha} dA(z) \\
        &=
      \int_{\mathbb{D}} |f(\zeta)|^{p(\zeta)}
            (1-|\phi_\lambda(\zeta)|^2)^{\alpha}
            |\phi_\lambda'(\zeta)|^2 \, dA(\zeta)\\
      &\leq
      C \int_{\mathbb{D}} |f(\zeta)|^{p(\zeta)}
            (1-|\zeta|^2)^{\alpha} \, dA(\zeta),
      \end{split}
    \]
    where we have used the substitution $\zeta = \phi_\lambda(z)$.
    Here $C$ may depend on $\lambda$. 
    Thus $C_{\phi_\lambda}$ is bounded from $H^{p(\cdot)}$ to
    $H^{p(\phi_\lambda(\cdot))}$.  Since
    $C_{\phi} = C_\omega C_{\phi_\lambda}$, the result
    follows. 
    \end{proof}

\section{Radial Exponents and $A^q$}

For the results of this section, we will assume that $p$ is
a radial exponent, 
and that $0 < q \leq p(z) < \infty$ for some constant exponent $q$. 
We will find a necessary and sufficient condition 
that $A^q = A^{p(\cdot)}$ (as sets).
Note that if $A^q = A^{p(\cdot)}$, then the two spaces 
have equivalent norms by the closed graph theorem.  By H\"{o}lder's 
inequality for variable exponent spaces we see that 
$A^{p(\cdot)} \subset A^q$,
so the only question is whether $A^q \subset A^{p(\cdot)}$. 

We will need the following theorem, which may be of interest in itself. 
\begin{theorem}\label{thm:incmult}
Suppose that $f(x)$ is a nonnegative measurable 
function on $[0,1]$.  The following are equivalent:
\begin{enumerate}
\item \label{conda} For any nonnegative
  increasing $g$ in $L^1$, we have 
$\ds \int_0^1 f(x) g(x) \, dx \leq C \int_0^1 g(x) \, dx$ for some constant 
$C$.
\item \label{condb} For all $0 \leq x < 1$ we have 
$\ds \frac{1}{x} \int_{1-x}^1 f(t) \, dt \leq C$ for some constant $C$.
\item \label{condc} For all $0 \leq x < 1$ we have 
$\ds \frac{2}{x} \int_{1-x}^{1-(x/2)} f(t) \, dt \leq C$ for some constant $C$.
\end{enumerate}
\end{theorem}
Note that condition \eqref{condb} may be thought of as saying that the 
maximal function of $f$ is finite at $1$.
\begin{proof}
  It is clear that \eqref{conda} implies \eqref{condb} by taking
  $g(x) = (1/x) \chi_{[1-x,1)}(x)$. 
  It is also clear that 
\eqref{condb} implies \eqref{condc}. 

Now suppose that \eqref{condc} holds. Let $g$ be a function as in 
\eqref{conda}.  
Then  
\[
\int_{1-2^{-n+1}}^{1-2^{-n}} f(x) g(x) \, dx \leq 
g(1-2^{-n}) \int_{1-2^{-n+1}}^{1-2^{-n}} f(x) \, dx \leq 
C 2^{-n} g(1-2^{-n}).
\]
So 
\[
\begin{split}
\int_0^1 f(x) g(x) \, dx \leq C \sum_{n=1}^\infty 2^{-n} g(1-2^{-n}) 
&\leq 2C \sum_{n=1}^\infty 2^{-n-1} g(1-2^{-n}) \\  &\leq 
2C \int_{1/2}^1 g(x) \, dx
\end{split}
\]
which shows that \eqref{conda} holds.
\end{proof}

For a radial exponent $p(\cdot)$ let 
\[
M_{p(\cdot)}^{p(\cdot)}(r,f) = 
\frac{1}{2\pi} \int_0^{2\pi} |f(re^{i\theta})|^{p(r)} \,
  d\theta .
\]
Then we have
\[
\int_{\mathbb{D}} |f(z)|^{p(z)} \, dA(z) = 
\int_0^1 M_{p(r)}^{p(r)}(r,f)\,  r \, dr.
\]

We need the following lemma on the growth of integral means.
Here $p$ can be thought of as a constant exponent, since the radius
is fixed. 
\begin{lemma}\label{lemma:growth}
Suppose that $\|f\|_{A^q} \leq 1$ and $q \leq p < \infty$.  
Then
\[
M_p^p(r,f) \leq  M_q^q(r,f) \frac{1}{(1-r)^{2(p-q)/q}}.
\]
\end{lemma}
\begin{proof}
  The
  hypothesis on $f$ implies that $M_{\infty}(r,f) \leq  (1-r)^{-2/q}$
  (see \cite{Dragan_point_eval}). 
Thus 
\[
\begin{split}
\frac{1}{2\pi} \int_0^{2\pi} |f(re^{i\theta})|^p \, d\theta  &\leq 
\left(\sup_{0 \leq \theta < 2 \pi} |f(re^{i\theta})|^{p-q} \right)
  \frac{1}{2\pi} \int_0^{2\pi} |f(re^{i\theta})|^q \, d\theta \\ &\leq 
\frac{1}{(1-r)^{2(p-q)/q}} M_q^q(r,f).
\end{split}
\]
\end{proof}

We now come to the main result. 
\begin{theorem}\label{thm:radial}
  Let $q>0$ be given and let $p(z)$ be a variable exponent on the disc with
  $q \leq p(z) \leq \esssup p < \infty$. Also assume that $p(z) = p(|z|)$
  for all $z$.  Then the following are equivalent. 
\begin{enumerate}
\item \label{condi} $A^q = A^{p(\cdot)}$.  
\item \label{condii} For all $f$ in $A^q$ with $\|f\|_{A^q} = 1$, there is a
  constant $K$ such that
\[
\int_0^1 M_{p(r)}^{p(r)}(r,f) \, r\, dr \leq K \int_0^1 M_q^q(r,f) \, r\, dr
\]
\item \label{condiii} For all $f$ in $A^q$ there is a constant $K$ such that  
\[ \int_0^1 \frac{1}{(1-r)^{2(p(r)-q)/q}} M_q^q(r,f)\, r \, dr
   \leq K \int_0^1 M_q^q(r,f)\, r \, dr. 
\]
\item \label{condiv}
There is a constant $K$ such that 
for any increasing left continuous nonnegative function $g$ we have 
\[\int_0^1 \frac{1}{(1-r)^{2(p(r)-q)/q}} g(r) \, dr \leq K \int_0^1 g(r) \, dr.\]
\item \label{condv}
There is a constant $K$ such that 
for all $x$ such that $0 < x \leq 1$, it is the case that
$\ds \frac{1}{x} \int_{1-x}^1 \frac{1}{(1-r)^{2(p(r)-q)/q}} \leq K.$
\item \label{condvi}
There is a constant $K$ such that 
for all $x$ such that $0 < x \leq 1$, it is the case that
$\ds \frac{2}{x} \int_{1-x}^{1-(x/2)} \frac{1}{(1-r)^{2(p(r)-q)/q}} \leq K.$
\item \label{condvii}
For some $a$ such that $aq > 2$ 
there is a constant $C'$
independent of $\lambda$ so that
\[
\int_{\mathbb{D}} |K_{\lambda,a,q}(z)|^{p(z)} \, dz \leq C',
\]
where $\ds
K_{\lambda,a,q}(z) = \frac{(1-|\lambda|^2)^{a-(2/q)}}
{(1-\overline{\lambda}z)^a}$.
\item \label{condviii}
  The previous statement holds for all $a$ with $aq > 2$
  (where $C'$ may depend on $a$ and $q$). 
\end{enumerate}
\end{theorem}

Note that the following estimate holds: 
\[
  I_{\alpha,\beta}(z) := \int_{\mathbb{D}}
  \frac{(1-|w|^2)^{\alpha}}{|1-z\overline{w}|^{2+\alpha + \beta}}
  \, dA(w)
  < C (1-|z|^2)^{-\beta}
\]
for $-1 < \alpha < \infty$ and $\beta > 0$ \cite[Theorem 1.7]{Zhu_Ap}.
This implies that the $A^q$ norm of $K_{\lambda,a,q}$ is bounded by a
constant depending on $q$ and $a$ but not on $\lambda$. 
\begin{proof}
  The fact that \eqref{condii} implies \eqref{condi} is by definition and the
  fact that $A^{p(\cdot)} \subset A^q$.  
The fact that \eqref{condi} implies \eqref{condii} is by the closed 
graph theorem. 
The fact that \eqref{condiii} implies \eqref{condii} is from 
Lemma \ref{lemma:growth}.  
It is clear that \eqref{condiv} implies \eqref{condiii} since 
integral means are increasing.  The equivalence of 
\eqref{condiv}, \eqref{condv} and \eqref{condvi} follows from 
Theorem \ref{thm:incmult}.
The fact that \eqref{condii} implies \eqref{condviii} is clear, and
\eqref{condviii} trivially implies \eqref{condvii}. 

We will be done if we can show that 
\eqref{condvii} implies \eqref{condvi}. Let $0<r<1$ and let
$K_r = K_{r,a,q}$ for some $a$ such that $aq > 2$. 
Note that 
\begin{equation}\label{eq:kernel_intmeanineq}
  \begin{split}
  M_p^p(\rho, K_r) &\asymp \frac{(1-r^2)^{(aq-2)p/q}}{(1-r \rho)^{ap-1}}
\\
M_q^q(\rho, K_r) &\asymp \frac{(1-r^2)^{a q-2}}{(1-r \rho)^{a q -1}}
\end{split}
\end{equation}
by a well known estimate \cite[p.\ 84, Lemma 3]{D_Hp}. 
In fact, by Lemma \ref{lemma:poissonbound}, for $s > 1$ one has 
\[
(1-r^2)^{1-s} \leq 
  \frac{1}{2\pi} \int_0^{2\pi} \frac{1}{|1-re^{i\theta}|^s} \, d\theta 
\leq \frac{\Gamma(s-1)}{\Gamma(s/2)^2} (1-r^2)^{1-s}.
\]
Because $a p$ and $a q$ are bounded away from $1$ and 
$\infty$, taking $s = a p$ or $s = a q$ in the above expression
shows that the implied constants
in \eqref{eq:kernel_intmeanineq} are independent of $r$ and 
$\rho$. 
So 
\[
M_{p(\rho)}^{p(\rho)}(\rho, K_r) \asymp M_q^q(\rho,K_r) 
   \frac{(1-r^2)^{(aq-2)[(p/q)-1]}}{(1-r \rho)^{ap-aq}}
\]
Let %
$r' = (r+1)/2$.  
For $r \leq \rho \leq r'$ we have that 
\[
\begin{split}
1-r'^2 &\leq 1 - r \rho \leq 1 - r^2 \leq 2(1-r'^2), \\
1-r'^2 &\leq 1 - \rho^2 \leq 1 - r^2 \leq 2(1-r'^2) .
\end{split}
\]
Thus we have that $1-r^2 \asymp 1-r\rho \asymp 1-\rho^2 \asymp 1-\rho$.
Therefore
\[
 \frac{(1-r^2)^{(aq-2)[(p/q)-1]}}{(1-r \rho)^{ap-aq}}
\asymp
\frac{1}{(1-\rho^2)^{2(p-q)/q}}
\]
where the implied constant is independent of $\rho$. 
We also have for $\rho \geq r$ that 
$M_q^q(\rho, K_r) \geq C(1-r)^{-1}$ for some 
constant $C$.  Thus
\[
\frac{2}{1-r} \int_{r}^{r'} \frac{1}{(1-\rho)^{2[p(\rho)-q]/q}} d \rho \leq 
 C \frac{2}{1-r} \int_{r}^{r'} \frac{(1-r^2)^{(aq-2)[p(\rho)/q-1]}}
     {(1-r\rho)^{ap(\rho)-aq}}
     d \rho.
 \]
 But this is at most
 \[
   \begin{split}
&\phantom{{}={}} C \frac{2}{1-r} \int_{r}^{r'} \frac{(1-r^2)^{(aq-2)[p(\rho)/q-1]}}
     {(1-r\rho)^{a p(\rho)-a q}}
     M_q^q(\rho,K_r)(1-r) d \rho \\
     &\leq
     C \int_r^{r'} M_{p(\rho)}^{p(\rho)}(\rho, K_r) \, d\rho \\
     &\leq
     C \int_r^{r'} M_{p(\rho)}^{p(\rho)}(\rho, K_r) \rho \, d\rho
     \leq C K.
   \end{split}
\] 
This finishes the proof.
\end{proof}

The following corollary is clear.
\begin{corollary}
Let $p(\cdot)$ be as in Theorem \ref{thm:radial}.  
If $(1-r)^{q-p(r)}$ is bounded then $A^q = A^{p(\cdot)}$.
\end{corollary}
This corollary implies that Theorem \ref{thm:logholder_equiv} holds in
the special case of log-H\"{o}lder continuous radial exponents,  but it
also applies in cases where that theorem does not. 

The following surprising corollary also follows. 
\begin{corollary}\label{cor:limsup_p_big}
Let $P > q$ be given.  There is a radial function $p(r) \geq q$ such that 
$\limsup_{0 \leq r < 1} p(r) = P$ and $A^{p(\cdot)} = A^q$. 
\end{corollary}
\begin{proof}
Define $p(r)$ so that for $n \geq 1$ we have 
\[
p(r) = \begin{cases} P \text{ if $1 - 2^{-n+1} \leq r < R_n$} \\ 
                     q \text{ if $R_n \leq r < 1-2^{-n}$}
\end{cases}
\]
where $R_n$ is chosen close enough to $1-2^{-n+1}$ so that 
\[
\int_{1-2^{-n+1}}^{R_n} (1-r)^{2(q-P)/q} \, dr < 1/2^n.
\]
Then 
\[
\int_{1-2^{-n+1}}^1 (1-r)^{2(q-p(r))/q} \, dr < 2^{-n+1},
\]
so for any $0 < x \leq 1$ we have 
\[
\frac{1}{x} \int_{1-x}^1 \frac{1}{(1-r)^{2(p(r)-q)/q}} \, dr < 2.
\]
\end{proof}

The following corollary also follows.  
\begin{corollary}
  For $0 <q < \infty$ there is a radial exponent $p(\cdot)$ such that
  $p(r) \rightarrow q$ as $r \rightarrow 1$ but that
  $A^{p(\cdot)} \neq A^{q}$.
\end{corollary}
\begin{proof}
  Let $p(r) = q + (q/2)[-\log(1-r)]^{-1/2}$.  Then
  \[
    \frac{1}{(1-r)^{2(p(r)-q)/q}} =
    \exp \left[\left( \log \frac{1}{1-r} \right)^{1/2}\right]
  \]
  which is an increasing function that is unbounded as $r \rightarrow 1$.
  Thus its average value on intervals of the form $[r,1)$ is unbounded. 
\end{proof}

\providecommand{\bysame}{\leavevmode\hbox to3em{\hrulefill}\thinspace}
\providecommand{\MR}{\relax\ifhmode\unskip\space\fi MR }
% \MRhref is called by the amsart/book/proc definition of \MR.
\providecommand{\MRhref}[2]{%
  \href{http://www.ams.org/mathscinet-getitem?mr=#1}{#2}
}
\providecommand{\href}[2]{#2}


\begin{thebibliography}{10}

\bibitem{MR0117349}
Lennart Carleson, \emph{An interpolation problem for bounded analytic
  functions}, Amer. J. Math. \textbf{80} (1958), 921--930. \MR{0117349}

\bibitem{MR0141789}
\bysame, \emph{Interpolations by bounded analytic functions and the corona
  problem}, Ann. of Math. (2) \textbf{76} (1962), 547--559. \MR{0141789}

\bibitem{Cowen-Maccluer_Comp}
Carl~C. Cowen and Barbara~D. MacCluer, \emph{Composition operators on spaces of
  analytic functions}, Studies in Advanced Mathematics, CRC Press, Boca Raton,
  FL, 1995. \MR{1397026}

\bibitem{DCU-VariableHardy}
David Cruz-Uribe and Li-An~Daniel Wang, \emph{Variable {H}ardy spaces}, Indiana
  Univ. Math. J. \textbf{63} (2014), no.~2, 447--493. \MR{3233216}

\bibitem{DCU_VarLebesgueBook}
David~V. Cruz-Uribe and Alberto Fiorenza, \emph{Variable {L}ebesgue spaces},
  Applied and Numerical Harmonic Analysis, Birkh\"auser/Springer, Heidelberg,
  2013, Foundations and harmonic analysis. \MR{3026953}

\bibitem{DunfordSchwartzI}
Nelson Dunford and Jacob~T. Schwartz, \emph{Linear {O}perators. {I}. {G}eneral
  {T}heory}, With the assistance of W. G. Bade and R. G. Bartle. Pure and
  Applied Mathematics, Vol. 7, Interscience Publishers, Inc., New York;
  Interscience Publishers, Ltd., London, 1958. \MR{0117523}

\bibitem{D_Hp}
Peter Duren, \emph{Theory of {$H\sp{p}$} spaces}, Pure and Applied Mathematics,
  Vol. 38, Academic Press, New York, 1970. \MR{0268655 (42 \#3552)}

\bibitem{D_Ap}
Peter Duren and Alexander Schuster, \emph{Bergman spaces}, Mathematical Surveys
  and Monographs, vol. 100, American Mathematical Society, Providence, RI,
  2004. \MR{2033762 (2005c:30053)}

\bibitem{D_Carlesonmeasure}
Peter~L. Duren, \emph{Extension of a theorem of {C}arleson}, Bull. Amer. Math.
  Soc. \textbf{75} (1969), 143--146. \MR{0241650}

\bibitem{tjf:bergprojbounds}
Timothy Ferguson, \emph{Bounds on integral means of bergman projections},
  Houston Mathematics Journal.

\bibitem{Zhu_Ap}
H{\aa}kan Hedenmalm, Boris Korenblum, and Kehe Zhu, \emph{Theory of {B}ergman
  spaces}, Graduate Texts in Mathematics, vol. 199, Springer-Verlag, New York,
  2000. \MR{1758653 (2001c:46043)}

\bibitem{MR0234002}
Lars H\"ormander, \emph{{$L^{p}$} estimates for (pluri-) subharmonic
  functions}, Math. Scand. \textbf{20} (1967), 65--78. \MR{0234002}

\bibitem{Kokilashvili-Paatashvili2006}
V.~Kokilashvili and V.~Paatashvili, \emph{On {H}ardy classes of analytic
  functions with a variable exponent}, Proc. A. Razmadze Math. Inst.
  \textbf{142} (2006), 134--137. \MR{2294576}

\bibitem{Kokilashvili-Paatashvili2008}
\bysame, \emph{On the convergence of sequences of functions in {H}ardy classes
  with a variable exponent}, Proc. A. Razmadze Math. Inst. \textbf{146} (2008),
  124--126. \MR{2464049}

\bibitem{Kokilashvili-Paatashvili2015}
\bysame, \emph{On variable exponent {H}ardy classes of analytic functions},
  Proc. A. Razmadze Math. Inst. \textbf{169} (2015), 93--103. \MR{3453827}

\bibitem{Dragan_point_eval}
Dragan Vukoti{\'c}, \emph{A sharp estimate for {$A^p_\alpha$} functions in
  {${\bf C}^n$}}, Proc. Amer. Math. Soc. \textbf{117} (1993), no.~3, 753--756.
  \MR{1120512 (93d:46042)}

\bibitem{Dragan_isoperimetric}
Dragan Vukoti\'c, \emph{The isoperimetric inequality and a theorem of {H}ardy
  and {L}ittlewood}, Amer. Math. Monthly \textbf{110} (2003), no.~6, 532--536.
  \MR{1984405}

\bibitem{XuanThesis}
Xuan Wang, In preparation, Thesis (Ph.D.)--University of Alabama.

\end{thebibliography}
\end{document}